\documentclass{amsart}
\usepackage{amsmath}
\usepackage{amssymb}
\usepackage{array}
\usepackage{amscd}
\usepackage{tikz-cd}

\numberwithin{equation}{section}
\newtheorem{Th}[subsection]{Theorem}
\newtheorem*{Th*}{Theorem}
\newtheorem{Lemma}[subsection]{Lemma}

\newtheorem{Cor}[subsection]{Corollary}
\theoremstyle{definition}

\newtheorem*{definition*}{Definition}
\newtheorem{Remark}[subsection]{Remark}

\newcommand{\comm}[1]{}

\usepackage{color}
\usepackage[normalem]{ulem}
\definecolor{DarkGreen}{rgb}{0,0.5,0.1} 

\newcommand\soutD{\bgroup\markoverwith
{\textcolor{DarkGreen}{\rule[.5ex]{2pt}{1pt}}}\ULon}

\makeatletter
\newcommand{\romsmall}[1]{\romannumeral #1}
\newcommand*{\rom}[1]{\expandafter\@slowromancap\romannumeral #1@}
\makeatother

\begin{document}
 
\title[The  automorphism group of del Pezzo surface of degree $2$]{The largest automorphism group of del Pezzo surface of degree $2$ without points}
\author{Anastasia V.~Vikulova}
\address{{\sloppy
\parbox{0.9\textwidth}{
Steklov Mathematical Institute of Russian
Academy of Sciences,
8 Gubkin str., Moscow, 119991, Russia
}\bigskip}}
\email{vikulovaav@gmail.com}
\date{}
\thanks{This work was supported by the Russian Science Foundation under grant no. 23-11-00033, https://rscf.ru/en/project/23-11-00033/}
\maketitle

\begin{abstract}
We construct an example  of a field and a smooth del Pezzo surface of degree~$2$ over this field  without points such that its automorphism group is isomorphic to $\mathrm{PSL}_2(\mathbb{F}_7) \times \mathbb{Z}/2\mathbb{Z},$ which is  the largest possible automorphism group of  del Pezzo surfaces of degree $2$ over an algebraically closed field of characteristic zero. 

\end{abstract}

\section{Introduction}

It is natural to expect  that for a given field~$\mathbf{F}$ a variety $X$  over~$\mathbf{F},$ which does not have $\mathbf{F}$-points, has much less automorphisms than $X_{\overline{\mathbf{F}}}$ has. Let us show some examples which demonstrate such a behavior.

Recall that a surface over a field $\mathbf{F}$ is called a \textit{ Severi--Brauer surface} if it is isomorphic to $\mathbb{P}^2$ over the algebraic closure of $\mathbf{F}.$  A Severi--Brauer surface without points is called  non-trivial Severi--Brauer surface.

\begin{Th}[{cf.~\cite[Theorem 1.3]{ShramovfgSV}}]\label{th:ShramovfgSV}
Let $G$ be a finite group. If there is a field $\mathbf{F}$ of characteristic zero and a non-trivial Severi--Brauer surface $S$ over $\mathbf{F}$ such that the group $\mathrm{Aut}(S)$ contains a subgroup isomorphic to $G,$ then there is $n \in \mathbb{N}$ divisible only by primes congruent to $1$ modulo $3$ such that $G$ is isomorphic to a subgroup of the group  $\mathbb{Z}/3\mathbb{Z} \times (\mathbb{Z}/n\mathbb{Z} \rtimes \mathbb{Z}/3\mathbb{Z}).$

\end{Th}

Therefore,  Theorem~\ref{th:ShramovfgSV} shows that  the set of finite subgroups in the automorphism group of a non-trivial Severi--Brauer surface over a field $\mathbf{F}$ of characteristic zero is \textit{strictly} contained  in the set of finite subgroups of the  automorphism group of $\mathbb{P}^2$ over the algebraic closure of $\mathbf{F}$.

Over the field of rational numbers $\mathbb{Q}$ we can say more about finite subgroups of the automorphism group of a non-trivial Severi--Brauer surface.

\begin{Th}[{\cite[Theorem 1.3]{VikulovaSB}}]
Let $V$ be a non-trivial Severi--Brauer surface over the field $\mathbb{Q}$ and let $G$ be a finite group. Then  $G$ is isomorphic to a subgroup of~$\mathrm{Aut}(V)$ if and only if~\mbox{$G \subset \mathbb{Z}/3\mathbb{Z}.$}
\end{Th}

A result of the same kind is known for smooth cubic surfaces without points.

\begin{Th}[{\cite[Theorem 1.4(\romsmall{1})]{Shramovcubic}}]
Let $\mathbf{F}$ be a field of characteristic zero and let~$S$ be a smooth cubic surface over $\mathbf{F}.$ Suppose that $S$ has no points over $\mathbf{F}.$ Then the group $\mathrm{Aut}(S)$ is abelian.

\end{Th}

However, if we consider smooth del Pezzo surfaces of degree $2,$ the situation will change dramatically.  Recall that the  largest automorphism group of del Pezzo surfaces of degree $2$ over an algebraically closed field of characteristic zero isomorphic to~\mbox{$\mathrm{PSL}_2(\mathbb{F}_7) \times \mathbb{Z}/2\mathbb{Z}$}  (see, for instance,~\cite[\S~8.7.3]{Dolgachev}). The goal of this paper is to prove the following theorem.

\begin{Th}\label{th:nopoints}
Let $S$ be a surface defined by the equation
$$
w^2+(x^4+y^4+z^4)-\frac{3}{2}(1-\sqrt{-7})(x^2y^2+x^2z^2+y^2z^2)=0
$$
\noindent over $\mathbb{Q}(\sqrt{-7})$ in the weighted projective space $\mathbb{P}(2,1,1,1)$ with  homogeneous coordinates~$w,x,y,z$ such that $\deg(w)=2$ and $\deg(x)=\deg(y)=\deg(z)=1.$ Then the following holds:
\begin{enumerate}
\renewcommand\labelenumi{\rm (\roman{enumi})}
\renewcommand\theenumi{\rm (\roman{enumi})}

\item\label{conditiondp2}
$S$ is a smooth del Pezzo surface of degree $2;$

\item\label{conditionAut}
one has $\mathrm{Aut}(S)\simeq \mathrm{PSL}_2(\mathbb{F}_7) \times \mathbb{Z}/2\mathbb{Z};$

\item\label{conditionPic}
$\mathrm{Pic}(S) \simeq \mathbb{Z};$

\item\label{conditionnopoint}
$S$ does not have rational points.


\end{enumerate}

\end{Th}

\textbf{Acknowledgment.} The author would like to express her deepest appreciation to C.~A.~Shramov for proposing interesting problems and for constant support. Also the author would like to thank the referee for valuable comments and remarks.

\section{Preliminaries}
In this section we give some auxiliary lemmas which we need in  the proof of the main theorem. Here and below by $\mathbb{Z}_2$ and by $\mathbb{Q}_2$ we denote $2$-adic integers and rational numbers, respectively.

\begin{Lemma}[{\cite[Chapter \rom{2}, Theorem 4]{Serrearithmetic}}]\label{lemma:Serremod2}
Let  $x=2^nu$ be an element of~$\mathbb{Q}_2^*,$ where~\mbox{$n \in \mathbb{Z}$} and~\mbox{$u \in \mathbb{Z}_2^*.$} For $x$ to be a square it is necessary and sufficient that $n$ is even and~\mbox{$u \equiv 1 \mod 8.$}
\end{Lemma}

\begin{Lemma}\label{lemma:sqrt-7}
For the field  $\mathbb{Q}_2$ we have $\sqrt{-7} \in \mathbb{Q}_2$ and the following congruence relation holds:
$$
\pm \sqrt{-7} \equiv \pm 181 \mod 128.
$$
\end{Lemma}

\begin{proof}
Indeed, by Lemma \ref{lemma:Serremod2} the number $\sqrt{-7}$ does belong to $\mathbb{Q}_2,$ because 
$$
-7 \equiv 1 \mod 8.
$$

\noindent Take $t \equiv \pm 181 \mod 128,$  we get
$$
t^2+7=181^2+7=2^{15} \equiv 0 \mod 128.
$$
\end{proof}

\begin{Lemma}\label{lemma:completeQ2}
There is a valuation $|\cdot|_v$ on  $\mathbb{Q}(\sqrt{-7})$ such that the completion~$\mathbb{Q}(\sqrt{-7})_v$ of   $\mathbb{Q}(\sqrt{-7})$ with respect to this valuation is isomorphic to~$\mathbb{Q}_2.$ Moreover, one can choose an isomorphism $\theta:\mathbb{Q}(\sqrt{-7})_{v} \xrightarrow{\sim} \mathbb{Q}_2$ so that
$$
\theta(\sqrt{-7}) \equiv 181 \mod 128
$$
\noindent and 
$$
\theta\left(\frac{3}{2}(1-\sqrt{-7})\right) \equiv -14 \mod 64.
$$ 
\end{Lemma}

\begin{proof}
Let $|\cdot|_2$ be a valuation of the field $\mathbb{Q}$ such that for $a \in \mathbb{Z}$ we have
$$
|a|_2=2^{-\mathrm{ord}_2(a)},
$$
\noindent where $\mathrm{ord}_2(a)$ is the highest power of $2$ dividing $a.$ Let us consider the completion of the field $\mathbb{Q}(\sqrt{-7})$ with respect to a valuation extending the valuation~$|\cdot|_2$ from~$\mathbb{Q}.$ By~\cite[Chapter \rom{2}, \S 10]{Frohlich} there are at most $[\mathbb{Q}(\sqrt{-7}):\mathbb{Q}]=2$ extensions of~$|\cdot|_2$ to~$\mathbb{Q}(\sqrt{-7}),$ say $|\cdot|_{v_i}.$ If $L_i$ is a completion of $\mathbb{Q}(\sqrt{-7})$ with respect to~$|\cdot|_{v_i},$ then we have 
$$
\mathbb{Q}(\sqrt{-7}) \otimes_{\mathbb{Q}}\mathbb{Q}_2 \simeq \prod_i L_i.
$$
\noindent Since  $-7$ is a square in $\mathbb{Q}_2$ by Lemma \ref{lemma:sqrt-7}, we obtain
\begin{equation}\label{eq:completionoffield}
\mathbb{Q}(\sqrt{-7}) \otimes_{\mathbb{Q}}\mathbb{Q}_2 \simeq \mathbb{Q}_2[T]/(T^2+7) \simeq \mathbb{Q}_2 \times \mathbb{Q}_2. 
\end{equation}
\noindent Therefore, we get $L_i \simeq \mathbb{Q}_2.$  By Lemma \ref{lemma:sqrt-7} one of the roots of the equation 
$$
t^2+7 \equiv 0 \mod 128
$$
 is 
$$
t_1=181=1+2^2+2^4+2^5+2^7.
$$
\noindent Therefore, we have the desired isomorphism and the following congruence relation
\begin{equation}\label{eq:3/2+}
\frac{3}{2}(1-t_1) \equiv \frac{3}{2}(1-1-2^2-2^4-2^5-2^7) \equiv -2-2^2-2^3=-14 \mod 64.
\end{equation}

\end{proof}

\begin{Remark}
From \eqref{eq:completionoffield} we get that there are two extensions of the valuation~$|\cdot|_2$ on $\mathbb{Q}$ to $\mathbb{Q}(\sqrt{-7}).$ The completions of $\mathbb{Q}(\sqrt{-7})$ with respect to these extending valuations are isomorphic to $\mathbb{Q}_2.$ If $|\cdot|_{v_1}$ and $|\cdot|_{v_2}$ are these two extensions and~\mbox{$\theta_1:\mathbb{Q}(\sqrt{-7})_{v_1} \xrightarrow{\sim} \mathbb{Q}_2$} and $\theta_2:\mathbb{Q}(\sqrt{-7})_{v_2} \xrightarrow{\sim} \mathbb{Q}_2$ are two isomorphisms, then
$$
\theta_1(\sqrt{-7})=\sqrt{-7} \equiv 181 \mod 128 \quad \text{and} \quad \theta_2(\sqrt{-7})=-\sqrt{-7} \equiv -181 \mod 128,
$$ 

\noindent as the Galois group $\mathrm{Gal}(\mathbb{Q}(\sqrt{-7})/\mathbb{Q})$ acts in this way.
\end{Remark}

\begin{Lemma}\label{lemma:solutionmodulo64}
Let
\begin{equation}\label{eq:dP64}
f(w,x,y,z)=w^2+(x^4+y^4+z^4)+14(x^2y^2+x^2z^2+y^2z^2).
\end{equation}
\noindent Assume that $w,x,y,z$ are integer numbers and at least one of them is odd. Then
\begin{equation}\label{eq:notequiv64}
f(w,x,y,z) \not\equiv 0 \mod 64.
\end{equation}

\end{Lemma}

\begin{proof}
The proof is by reductio ad absurdum. Hence,  note that  $w^2+x^4+y^4+z^4$ should be even. So we have~$3$ possibilities: 
\begin{enumerate}
\item $w,x,y,z$ are odd;
\item $w$ is odd, one of $x$, $y$ and $z$ is odd and the other two are even;
\item $w$ is even, one of $x$, $y$ and $z$ is even and the other two are odd.
\end{enumerate}

\noindent In the first and in the second cases we get that $f(w,x,y,z) \equiv \pm 2 \mod 8,$ so that~$f(w,x,y,z) \not\equiv 0 \mod 64.$  In the third case without loss of generality we assume that $x$ is even and $y$ and $z$ are odd. So we write
\begin{equation}\label{eq:wxyz}
w=2\tilde{w}; \quad x=2\tilde{x}; \quad y=1+2\tilde{y}; \quad z=1+2\tilde{z},
\end{equation}

\noindent where $\tilde{w}, \tilde{x}, \tilde{y}, \tilde{z} \in \mathbb{Z}.$ Let us rewrite the right hand side of \eqref{eq:dP64} as 
$$
w^2+x^2(x^2+14y^2)+y^2(y^2+14z^2)+z^2(z^2+14x^2).
$$
\noindent Let us reduce it modulo $64.$ Using \eqref{eq:wxyz} we obtain
\begin{gather*}
x^2(x^2+14y^2) = 4\tilde{x}^2(4\tilde{x}^2+14+14\cdot 4\tilde{y}(\tilde{y}+1)) \equiv 16\tilde{x}^4+2^3 \cdot 7\tilde{x}^2 \mod 64;\\
y^2(y^2+14z^2) = (1+4\tilde{y}(\tilde{y}+1))(1+4\tilde{y}(\tilde{y}+1)+14+14\cdot 4\tilde{z}(\tilde{z}+1)) \equiv \\
 \equiv 15+2^3\cdot 7\tilde{z}(\tilde{z}+1) \mod 64;\\
z^2(z^2+14x^2) = (1+4\tilde{z}(\tilde{z}+1))(1+4\tilde{z}(\tilde{z}+1)+2^3\cdot 7\tilde{x}^2) \equiv\\
\equiv 1+2^3 \cdot 7\tilde{x}^2+2^3\tilde{z}(\tilde{z}+1) \mod 64.\\
\end{gather*}

\noindent From these three congruences we get that 
\begin{equation}\label{eq:final64}
\begin{gathered}
w^2+x^2(x^2+14y^2)+y^2(y^2+14z^2)+z^2(z^2+14x^2) \equiv\\
 \equiv w^2+16+2^4\cdot 7\tilde{x}^2+16\tilde{x}^4 \mod 64.
\end{gathered}
\end{equation}

\noindent So if $w \not\equiv 0 \mod 4,$ then the value~\eqref{eq:final64} is not congruent to~$0$ modulo~$64.$ Thus, assume that $w \equiv 0 \mod 4,$ i.e. $w=4\tilde{w},$ where $\tilde{w} \in \mathbb{Z}.$ It is enough to reduce
\begin{equation}\label{eq:mod4}
\tilde{w}^2+1+3\tilde{x}^2+\tilde{x}^4  
\end{equation}
\noindent modulo $4.$ We have 
$$
3\tilde{x}^2+\tilde{x}^4 \equiv 0  \mod 4
$$
\noindent and
$$
\tilde{w}^2+1 \equiv  1 \quad \text{or} \quad 2 \mod 4.
$$
\noindent Therefore, the value of~\eqref{eq:mod4} is not congruent to~$0$ modulo~$4.$ As a result, the right hand side of \eqref{eq:dP64} is not congruent to~$0$ modulo $64,$ as was to be shown.

\end{proof}

\begin{Cor}\label{cor}
The equation 
\begin{equation}\label{eq:equationincorolary}
w^2+(x^4+y^4+z^4)+14(x^2y^2+x^2z^2+y^2z^2)=0
\end{equation}
\noindent does not have  non-zero solutions in $\mathbb{Q}_2.$
\end{Cor}

\begin{proof}
Assume the contrary. Then without loss of generality we can look for a solution of \eqref{eq:equationincorolary} over $\mathbb{Z}_2$ with the condition that one of $w,x,y,z$ is odd.   Indeed, if~$(2w,2x,2y,2z)$ is a solution of \eqref{eq:equationincorolary}, we get
\begin{gather*}
4w^2+16(x^4+y^4+z^4)+16\cdot14(x^2y^2+x^2z^2+y^2z^2)=\\
=4(w^2+4(x^4+y^4+z^4)+4\cdot14(x^2y^2+x^2z^2+y^2z^2))=0.
\end{gather*}
\noindent Therefore, $w^2$ is divisible by $4.$ Thus, we have $w=2\tilde{w}$ for $\tilde{w} \in \mathbb{Z}_2$ and 
$$
(2w,2x,2y,2z)=(4\tilde{w},2x,2y,2z).
$$
\noindent It is a solution if and only if $(\tilde{w},x,y,z)$ is a solution.

So let $(w,x,y,z)$ be such solution of \eqref{eq:equationincorolary}. Then 
$$
(w,x,y,z) \mod 64
$$
\noindent is a solution of \eqref{eq:equationincorolary} modulo $64.$ However, by Lemma \ref{lemma:solutionmodulo64} this equation does not have a non-zero solution modulo $64$ such that at least one of the coordinates $w,x,y,z$ is odd. This contradiction proves the corollary.

\end{proof}

\begin{Lemma}\label{lemma:centralizer}
The group $\mathrm{PSL}_2(\mathbb{F}_7)$ lies in $W(\mathrm{E}_7),$ it is a simple group and its centralizer in  $W(\mathrm{E}_7)$ is isomorphic to $\mathbb{Z}/2\mathbb{Z}.$
\end{Lemma}

\begin{proof}
First of all, note that 
\begin{equation}\label{eq:simeqweyl}
W(\mathrm{E}_7) \simeq \mathbb{Z}/2\mathbb{Z} \times \mathrm{Sp}_6(\mathbb{F}_2)
\end{equation}
\noindent by~\cite[\S 3.12.4]{Wil}. We have the natural inclusion $\mathrm{PSL}_3(\mathbb{F}_2) \subset \mathrm{Sp}_6(\mathbb{F}_2),$ which maps the matrix $M \in \mathrm{PSL}_3(\mathbb{F}_2)$ to the matrix
$$
\begin{pmatrix}
M & 0\\
0 & (M^{-1})^{\mathrm{T}}\\
\end{pmatrix}
\in \mathrm{Sp}_6(\mathbb{F}_2).
$$
\noindent As there is an isomorphism $\mathrm{PSL}_2(\mathbb{F}_7) \simeq \mathrm{PSL}_3(\mathbb{F}_2)$ by~\cite[\S 3.12.1]{Wil}, we obtain the inclusion 
$$
\mathrm{PSL}_2(\mathbb{F}_7) \subset W(\mathrm{E}_7)
$$
\noindent by \eqref{eq:simeqweyl}. The group $\mathrm{PSL}_2(\mathbb{F}_7)$ is a simple group by~\cite[\S 3.3.1]{Wil}.

Let us consider a subgroup in $W(\mathrm{E}_7)$ isomorphic to $\mathbb{Z}/7\mathbb{Z}$ such that it  also lies in the subgroup $\mathrm{PSL}_2(\mathbb{F}_7)$ of the Weyl group.  By~\cite[Table 1]{Trepalin} elements of order~$7$ correspond to the conjugacy class $A_6$ in $W(\mathrm{E}_7)$ and by~\cite[Table 10]{Carter} the cardinality of the set of this conjugacy class is equal to $2^9 \cdot 3^4 \cdot 5.$ As 
$$
|W(\mathrm{E}_7)|=2^{10}\cdot 3^4 \cdot 5 \cdot 7,
$$
\noindent we get that the stabilizer (with respect to conjugation) of an element of order~$7$ in $W(\mathrm{E}_7)$ is of order $14.$ This means that the centralizer of $\mathbb{Z}/7\mathbb{Z}$ in $W(\mathrm{E}_7)$ is isomorphic to~$\mathbb{Z}/14\mathbb{Z}.$ Since $\mathrm{PSL}_2(\mathbb{F}_7)$ does not have a subgroup of order $14,$ this also means that the centralizer of $\mathbb{Z}/7\mathbb{Z} \subset \mathrm{PSL}_2(\mathbb{F}_7)$ lies in 
$$
\mathbb{Z}/2\mathbb{Z} \times\mathrm{PSL}_2(\mathbb{F}_7) \subset \mathbb{Z}/2\mathbb{Z} \times \mathrm{Sp}_6(\mathbb{F}_2).
$$
\noindent  But $\mathrm{PSL}_2(\mathbb{F}_7)$ is a simple group. Therefore, its centralizer in the Weyl group $W(\mathrm{E}_7)$ is isomorphic to $\mathbb{Z}/2\mathbb{Z}.$

\end{proof}

\section{Proof of Theorem~\ref{th:nopoints}}

In this section we prove Theorem \ref{th:nopoints}.

\begin{proof}[Proof of Theorem \ref{th:nopoints}]
Let us prove \ref{conditiondp2}. The smoothness of the surface $S$ can be proved by direct computation of partial derivatives. By~\cite[Proposition 6.3.6]{Dolgachev} the surface $S$ is a del Pezzo surface of degree $2.$

Let us prove \ref{conditionAut}. It is well-known that the anticanonical linear system of $S$ gives a double cover $\phi_{|-K_S|}:S \to \mathbb{P}^2$ which is ramified over a smooth quartic curve $C$. In our case the equation of the quartic $C$ is 
\begin{equation}\label{eq:quarticcurve}
(x^4+y^4+z^4)-\frac{3}{2}(1-\sqrt{-7})(x^2y^2+x^2z^2+y^2z^2)=0
\end{equation}
\noindent in $\mathbb{P}^2$ with homogeneous coordinates $x,y,z.$ The Geiser involution~$\gamma,$ which is a deck transformation of the cover  $\phi_{|-K_S|}$ and  acts on $S$ as 
$$
\gamma:(w,x,y,z) \mapsto (-w,x,y,z),
$$
\noindent gives us a non-trivial element in~$\mathrm{Aut}(S).$

 From~\cite[Chapter 1]{Elkies} we get that $\mathrm{Aut}(C) \subset \mathrm{SL}_3(\mathbb{Q}(\sqrt{-7}))$ is generated by the signed permutation matrices 
$$
-\begin{pmatrix}
1 & 0 & 0\\
0 & 0 & 1\\
0 & 1 & 0\\
\end{pmatrix}, \;\;
\begin{pmatrix}
0 & 1 & 0\\
0 & 0 & 1\\
1 & 0 & 0\\
\end{pmatrix}, \;\;
\begin{pmatrix}
0 & 1 & 0\\
1 & 0 & 0\\
0 & 0 & -1\\
\end{pmatrix}
$$

\noindent and the matrix
$$
\frac{1}{2}\begin{pmatrix}
-1 & 1 & \frac{-1-\sqrt{7}}{2}\\
\frac{-1+\sqrt{7}}{2} & \frac{-1+\sqrt{7}}{2} & 0\\
-1 & 1 & \frac{1+\sqrt{7}}{2}\\
\end{pmatrix}
$$
which act on the equation of~$C$ by the identity. This gives us the natural embedding 
$$
\mathrm{PSL}_2(\mathbb{F}_7) \subset \mathrm{Aut}(S).
$$
\noindent Since  the Geiser involution~\mbox{$\gamma \in \mathrm{Aut}(S)$} and any element in $\mathrm{PSL}_2(\mathbb{F}_7) \subset \mathrm{Aut}(S)$ commute with each other, we get the embedding
\begin{equation}\label{eq:subsetbutisom}
\mathrm{PSL}_2(\mathbb{F}_7) \times \mathbb{Z}/2\mathbb{Z} \subset \mathrm{Aut}(S).
\end{equation}

\noindent By~\cite[\S~8.7.3]{Dolgachev} we get that $\mathrm{PSL}_2(\mathbb{F}_7) \times \mathbb{Z}/2\mathbb{Z}$ is the largest  automorphism group of a del Pezzo surface of degree~$2$ over an algebraically closed field of characteristic zero,  so~\eqref{eq:subsetbutisom}   is an isomorphism, as desired.

\comm{

  The double cover $\phi_{|-K_S|}$ gives us the following exact sequence
\begin{equation}\label{eq:exactsequensecanonical}
0 \to \mathbb{Z}/2\mathbb{Z} \to \mathrm{Aut}(S) \to \mathrm{Aut}(\mathbb{P}^2),
\end{equation}
\noindent where $\mathbb{Z}/2\mathbb{Z}$ is generated by the Geiser involution~$\gamma$ which is a deck transformation of the cover  $\phi_{|-K_S|}$ and  acts on $S$ as 
$$
\gamma:(w,x,y,z) \mapsto (-w,x,y,z).
$$
\noindent From \eqref{eq:exactsequensecanonical} we obtain
\begin{equation}\label{eq:aut}
\mathrm{Aut}(S)/\left(\mathbb{Z}/2\mathbb{Z}\right) \simeq \mathrm{Aut}(C) \simeq \mathrm{PSL}_2(\mathbb{F}_7),
\end{equation}
\noindent where the last isomorphism follows  from~\cite[Chapter 1]{Elkies}.  From~\cite[Chapter 1]{Elkies} we get that $\mathrm{Aut}(C) \subset \mathrm{SL}_3(\mathbb{Q}(\sqrt{-7}))$ is generated by the signed permutation matrices 
$$
-\begin{pmatrix}
1 & 0 & 0\\
0 & 0 & 1\\
0 & 1 & 0\\
\end{pmatrix}, \;\;
\begin{pmatrix}
0 & 1 & 0\\
0 & 0 & 1\\
1 & 0 & 0\\
\end{pmatrix}, \;\;
\begin{pmatrix}
0 & 1 & 0\\
1 & 0 & 0\\
0 & 0 & -1\\
\end{pmatrix}
$$

\noindent and the matrix
$$
\frac{1}{2}\begin{pmatrix}
-1 & 1 & \frac{-1-\sqrt{7}}{2}\\
\frac{-1+\sqrt{7}}{2} & \frac{-1+\sqrt{7}}{2} & 0\\
-1 & 1 & \frac{1+\sqrt{7}}{2}\\
\end{pmatrix}
$$
which act on the equation of~$C$ by the identity. This gives us the natural embedding 
$$
\mathrm{PSL}_2(\mathbb{F}_7) \subset \mathrm{Aut}(S).
$$
\noindent Since  the Geiser involution~\mbox{$\gamma \in \mathrm{Aut}(S)$} and any element in $\mathrm{PSL}_2(\mathbb{F}_7) \subset \mathrm{Aut}(S)$ commute with each other, we get the desired isomorphism
$$
\mathrm{PSL}_2(\mathbb{F}_7) \times \mathbb{Z}/2\mathbb{Z} \simeq \mathrm{Aut}(S).
$$

}

Let us prove \ref{conditionPic}. To prove that $\mathrm{Pic}(S) \simeq \mathbb{Z}$ one should prove that  
\begin{equation}\label{eq:picgaloisinvariant}
\mathrm{Pic}(\bar{S})^{\mathrm{Gal}(\bar{\mathbb{Q}}/\mathbb{Q}(\sqrt{-7}))} \simeq \mathbb{Z},
\end{equation}
\noindent where $\bar{S}=S \times_{\mathrm{Spec}(\mathbb{Q}(\sqrt{-7}))}\mathrm{Spec}(\bar{\mathbb{Q}}),$ because we have the obvious inclusion 
$$
\mathrm{Pic}(S) \subseteq \mathrm{Pic}(\bar{S})^{\mathrm{Gal}(\bar{\mathbb{Q}}/\mathbb{Q}(\sqrt{-7}))}.
$$
\noindent To prove \eqref{eq:picgaloisinvariant} note that the image of the Galois group $\mathrm{Gal}(\bar{\mathbb{Q}}/\mathbb{Q}(\sqrt{-7}))$ in the Weyl group $W(\mathrm{E}_7)$ and the automorphism group $\mathrm{Aut}(S),$ which lies in $W(\mathrm{E}_7)$ by~\cite[Corollary 8.2.40]{Dolgachev}, commute with each other. So we need to consider the centralizer of~\mbox{$\mathrm{Aut}(S) \simeq \mathrm{PSL}_2(\mathbb{F}_7) \times \mathbb{Z}/2\mathbb{Z}$} in $W(\mathrm{E}_7).$ By Lemma \ref{lemma:centralizer} it is isomorphic to~$\mathbb{Z}/2\mathbb{Z}$ and it is obvious that it is generated by the Geiser involution. This means that the image of the Galois group $\mathrm{Gal}(\bar{\mathbb{Q}}/\mathbb{Q}(\sqrt{-7}))$ in the Weyl group $W(\mathrm{E}_7)$ is either trivial, or isomorphic to $\mathbb{Z}/2\mathbb{Z} \simeq \langle \gamma \rangle.$ 

To prove that it is not trivial it is enough to show that that there is an exceptional curve on $S$ which is not defined over $\mathbb{Q}(\sqrt{-7}).$ According to the basic theory of smooth del Pezzo surfaces of degree $2$ (see, for instance,~\cite[\S 8.7.1]{Dolgachev}) the surface $S$ has $56$ lines which are the inverse image of $28$ bitangents of the branch curve $C.$ It is not hard to see that the line in $\mathbb{P}^2$ defined by the equation
$$
x+y+z=0
$$
is a bitangent of $C.$ So the inverse image of this bitangent are two lines which are defined by the equations
$$
w \pm i\left(\frac{3+\sqrt{-7}}{2} \right)(x^2+xy+y^2)=0.
$$
\noindent They are not defined over $\mathbb{Q}(\sqrt{-7}),$ hence, the image of  $\mathrm{Gal}(\bar{\mathbb{Q}}/\mathbb{Q}(\sqrt{-7}))$ in the Weyl group $W(\mathrm{E}_7)$ is isomorphic to $\mathbb{Z}/2\mathbb{Z} \simeq \langle \gamma \rangle.$ As it is generated by the Geiser involution which  interchanges the two lines in the inverse image of
any bitangent of $C$, we get
$$ 
\mathrm{Pic}(\bar{S})^{\mathrm{Gal}(\bar{\mathbb{Q}}/\mathbb{Q}(\sqrt{-7}))} \simeq \mathbb{Z}.
$$

Let us prove \ref{conditionnopoint}. To prove that $S$ does not have a rational point over the field~$\mathbb{Q}(\sqrt{-7})$ it is enough to prove that  it does not have a rational point over some completion of $\mathbb{Q}(\sqrt{-7}).$ Let us consider the valuation $|\cdot|_2$ on $\mathbb{Q}.$  Let $|\cdot|_v$ be one of the extensions of $|\cdot|_2$ on $\mathbb{Q}(\sqrt{-7})$ (which was defined in Lemma~\ref{lemma:completeQ2}). Then by Lemma~\ref{lemma:completeQ2} for the completion~$\mathbb{Q}(\sqrt{-7})_v$ of $\mathbb{Q}(\sqrt{-7})$ with respect to $|\cdot|_v$ we have an isomorphism
$$
\theta:\mathbb{Q}(\sqrt{-7})_v  \xrightarrow{\sim} \mathbb{Q}_2.
$$ 
This isomorphism satisfies
$$
\theta\left(\frac{3}{2}(1-\sqrt{-7})\right) \equiv -14 \mod 64.
$$ 

\noindent So to prove that  $S$ does not have a rational point over the field $\mathbb{Q}_2$ it is enough to prove that
$$
w^2+(x^4+y^4+z^4)+14(x^2y^2+x^2z^2+y^2z^2)=0
$$
\noindent does not have a solution over $\mathbb{Q}_2.$   However, by Corollary \ref{cor} it does not have a non-zero solution. Quod erat demonstrandum.

\end{proof}

%
%
%

\providecommand{\bysame}{\leavevmode\hbox to3em{\hrulefill}\thinspace}
\providecommand{\MR}{\relax\ifhmode\unskip\space\fi MR }
\providecommand{\MRhref}[2]{%
  \href{http://www.ams.org/mathscinet-getitem?mr=#1}{#2}
}
\providecommand{\href}[2]{#2}

\end{document}